\documentclass{article}

\usepackage{rotating}
\usepackage{tikz-cd}
\usepackage[all]{xy}
\usepackage{amssymb}
\usepackage{amsmath}
\usepackage{amsthm}
\newtheorem{thm}{Theorem}
\newtheorem{prop}{Proposition}
\newtheorem{lem}{Lemma}

\newtheorem{defiprop}{Definition-Proposition}

\def\Top{\mathop{\rm Top}\nolimits}
\def\RTop{\mathop{\rm RTop}\nolimits}
\def\Sch{\mathop{\rm Sch}\nolimits}

\def\codim{\mathop{\rm codim}\nolimits}
\def\dim{\mathop{\rm dim}\nolimits}

\def\SmVar{\mathop{\rm SmVar}\nolimits}
\def\Gr{\mathop{\rm Gr}\nolimits}

\def\PSmVar{\mathop{\rm PSmVar}\nolimits}
\def\Var{\mathop{\rm Var}\nolimits}
\def\Hom{\mathop{\rm Hom}\nolimits}
\def\Spec{\mathop{\rm Spec}\nolimits}

\def\QPVar{\mathop{\rm QPVar}\nolimits}
\def\AnSp{\mathop{\rm AnSp}\nolimits}
\def\CW{\mathop{\rm CW}\nolimits}

\def\PVar{\mathop{\rm PVar}\nolimits}

\def\sing{\mathop{\rm sing}\nolimits}

\def\Im{\mathop{\rm Im}\nolimits}

\def\Cone{\mathop{\rm Cone}\nolimits}
\def\ad{\mathop{\rm ad}\nolimits}

\def\log{\mathop{\rm log}\nolimits}
\def\Diff{\mathop{\rm Diff}\nolimits}

\def\PSh{\mathop{\rm PSh}\nolimits}
\def\AnSm{\mathop{\rm AnSm}\nolimits}

\def\Fun{\mathop{\rm Fun}\nolimits}

\def\Ext{\mathop{\rm Ext}\nolimits}

\def\Cat{\mathop{\rm Cat}\nolimits}
\def\RCat{\mathop{\rm RCat}\nolimits}
\def\TriCat{\mathop{\rm TriCat}\nolimits}

\def\Shv{\mathop{\rm Shv}\nolimits}

\def\Vect{\mathop{\rm Vect}\nolimits}

\def\PSch{\mathop{\rm PSch}\nolimits}

\def\dar[#1]{\ar@<2pt>[#1]\ar@<-2pt>[#1]}

\title{Degeneration of families of projective hypersurfaces and Hodge conjecture}

\author{Johann Bouali}

\begin{document}

\maketitle

\begin{abstract}
We prove by induction on dimension the Hodge conjecture for smooth complex projective varieties. 
Let $X$ be a smooth complex projective variety. Then $X$ is birational to a possibly singular projective hypersurface,
hence to a smooth projective variety $E_0$ which is a component of a normal crossing divisor $E=\cup_{i=0}^rE_i\subset Y$
which is the singular fiber of a pencil $f:Y\to\mathbb A^1$ of smooth projective hypersurfaces.
Using the smooth hypersurface case (\cite{B8} theorem 1),
the nearby cycle functor on mixed Hodge module with rational de Rham factor and the induction hypothesis,
we prove that a Hodge class of $E_0$ is absolute Hodge, more precisely 
the locus of Hodge classes inside the algebraic vector bundle given the De Rham cohomology of the rational deformation of $E_0$ is 
defined over $\mathbb Q$. By \cite{B7} theorem 4, we get the Hodge conjecture for $E_0$. 
By the induction hypothesis we also have the Hodge conjecture for $X$.  
\end{abstract}

\tableofcontents

\section{Introduction}

The main result of this article is an inductive proof of the Hodge conjecture using the results of \cite{B7} and \cite{B8}.
More precisely, assuming the Hodge conjecture for smooth complex projective varieties of dimension less or equal to $d-1$, 
the property of satisfying the Hodge conjecture for connected smooth complex projective varieties of dimension $d$
is a birational invariant (proposition \ref{birHdg}). 
Let $X$ be a connected smooth complex projective variety of dimension $d$.
then $X$ is birational to a possibly singular hypersurface $Y_{0_X}\subset\mathbb P^{d+1}_{\mathbb C}$
of degree $l_X:=deg(\mathbb C(X)/\mathbb C(x_1,\cdots,x_d))$. Consider
\begin{equation*}
f:\mathcal Y'_{\mathbb C}=V(\tilde f)\subset\mathbb P^{d+1}_{\mathbb C}\times\mathbb P(l_X)\to\mathbb P(l_X)
\end{equation*}
the universal family of projective hypersurfaces of degree $l_X$, 
and $0_X\in\mathbb P(l_X)$ the point corresponding to $Y_{0_X}$. 
Denote $\Delta\subset\mathbb P(l_X)$ the discriminant locus parametrizing the hypersurfaces which are singular.
Let $\bar S:=\bar{0_X}^{\mathbb Q}\subset\mathbb P(l_X)$ the $\bar{\mathbb Q}$-Zariski closure of $0_X$ inside $\mathbb P(l_X)$.
If $Y_{0_X}$ is smooth, then $X$ satisfy the hodge conjecture since $Y_{0_X}$ satisfy the Hodge conjecture by \cite{B8}
and $X$ is birational to $Y_{0_X}$, using the induction hypothesis.
Hence, we may assume that $Y_{0_X}$ is singular, that is $0_X\in\Delta$. 
Then, $\bar S\subset\Delta$ since $\Delta\subset\mathbb P(l_X)$ is defined over $\mathbb Q$.
Let $\bar T\subset\mathbb P(l_X)$ be an irreducible closed subvariety of dimension $\dim(\bar S)+1$ defined over $\bar{\mathbb Q}$
such that 
\begin{equation*}
\bar S\subset\bar T \; \mbox{and} \; T:=\bar T\cap(\mathbb P(l_X)\backslash\Delta)\neq\emptyset.
\end{equation*}
We take for $\bar T=V(f_1,\cdots,f_s)\subset\mathbb P(l_X)$ 
an irreducible complete intersection defined over $\bar{\mathbb Q}$ which intersect $\Delta$ properly and such that 
\begin{itemize}
\item $V(f_1)\supset\sing(\Delta)\cup\bar S$ 
\item for each $r\in[1,\cdots,s]$, $V(f_1,\cdots,f_r)\supset\sing(V(f_1,\cdots,f_{r-1})\cap\Delta)\cup\bar S$ 
contain inductively the singular locus of the intersection with $\Delta$ and $\bar S$.
\end{itemize}
Then by the weak Lefchetz hyperplane theorem for homotopy groups,
$T\subset\mathbb P(l_X)\backslash\Delta$ is not contained in a weakly special subvariety since 
\begin{equation*}
i_{T*}:\pi_1(T_{\mathbb C})\to\pi_1(\mathbb P(l_X)\backslash\Delta), \; \;  (\Delta=\Delta_{\mathbb C}) 
\end{equation*}
is an isomorphism if $\dim\bar T\geq 2$ (i.e. $\dim\bar S\geq 1)$ and surjective if $\dim\bar T=1$ (i.e. $\dim\bar S=0$).
We have then $\bar T\cap\Delta=\bar S\cup\bar S'$, where $\bar S'\subset\bar T$ is a divisor.
We then consider the family of algebraic varieties defined over $\bar{\mathbb Q}$
\begin{equation*}
\tilde f_X:\mathcal Y_{\bar T^o}\xrightarrow{\epsilon}\mathcal Y'_{\bar T^o}
\xrightarrow{f\times_{\mathbb P(l_X)}\bar T^o}\bar T^o 
\end{equation*}
where 
\begin{itemize}
\item $\tilde f_X:\mathcal Y\xrightarrow{\epsilon}\mathcal Y'_{\bar T}\xrightarrow{f\times_{\mathbb P(l_X)}\bar T}\bar T$ 
with $\epsilon:(\mathcal Y,\mathcal E')\to(\mathcal Y'_{\bar T},\mathcal Y'_{\bar S}\cup\mathcal Y'_{\bar S'})$ 
is a desingularization over $\mathbb Q$, in particular $\mathcal E'=\mathcal E\cup\mathcal E''$
with $\mathcal E=\cup_{i=0}^r\mathcal E_i:=\tilde f_X^{-1}(\bar S)\subset\mathcal Y$
which is a normal crossing divisor, in particular $\mathcal Y_T=\mathcal Y'_T$,
\item $S\subset\bar S$ is the open subset such that $S$ is smooth and 
$\tilde f_{X}\times_{\bar T}S:\mathcal E_I:=\cap_{i\in I}\mathcal E_i\to S$ is 
smooth projective for each $I\subset[0,\ldots,r]$,
$\bar T^o\subset\bar T$ is an open subset such that $\bar T^o\cap\Delta=S$ and $T^o:=\bar T^o\cap T$ is smooth,
for simplicity we denote again $\mathcal E$ for the open subset $\mathcal E_S:=\tilde f_X^{-1}(S)\subset\mathcal E$,
\item $t\in S(\mathbb C)$ is the point such that $\mathcal Y'_{\mathbb C,t}=Y_{0_X}$ and 
$\epsilon_t:E_0:=\mathcal E_{0,\mathbb C,t}\to Y_{0_X}$ is surjective 
(there is at least one component of $E$ dominant over $Y_{0_X}$ since $E$ is dominant over $Y_{0_X}$ as
$\mathcal E$ is dominant over $\mathcal Y'_{\bar S}$), 
in particular $E_0$ is birational to $Y_{0_X}$, hence $E_0$ is birational to $X$. 
We denote $E_i:=\mathcal E_{i,\mathbb C,t}$ for each $0\leq i\leq r$ and $E:=\mathcal E_{\mathbb C,t}=\cup_{i=0}^rE_i$.
\end{itemize} 
Using \cite{B7} and \cite{B8},
assuming the Hodge conjecture for smooth projective varieties of dimension less or equal to $d-1$, 
we then prove the Hodge conjecture for the $E_i$, more specifically for $E_0$ as follows :
Let $p\in\mathbb Z$. By the hard Lefchetz theorem on $H^*(E_0)$, we may assume that $2p\leq d$.
We then consider (see section 2.4) the morphism of mixed Hodge module over $S$ 
(in fact of geometric variation of mixed Hodge structures over $S$) given by the specialization map
\begin{eqnarray*} 
Sp_{\mathcal E/\tilde f_X}:E^{2p}_{Hdg}(\mathcal E/S)\to 
H^{2p}\psi_SE_{Hdg}(\mathcal Y_{T^o}/T^o)=\psi_SE^{2p}_{Hdg}(\mathcal Y_{T^o}/T^o),
\end{eqnarray*}
whose de Rham part are morphism of filtered vector bundle over $S$ with Gauss-Manin connexion defined over $\bar{\mathbb Q}$.
Moreover, we have, see lemma \ref{WSp}, an isomorphism of variation of mixed Hodge structures over $S$
\begin{eqnarray*}
c_{\mathcal E}:=H^{2p-1}(\tilde f_X\circ i_{\mathcal E})_*c(\mathbb Q^{Hdg}_{\mathcal E})\circ c_V:
E_{Hdg,\mathcal E}^{2p-1}(\mathcal Y_{\bar T^o}/\bar T^o)\xrightarrow{\sim} \\
\phi_SE_{Hdg,\mathcal E}^{2p-1}(\mathcal Y_{T^o}/T^o)\xrightarrow{\sim} 
\ker(Sp_{\mathcal E/\tilde f_X}:E^{2p}_{Hdg}(\mathcal E/S)\to\psi_SE^{2p}_{Hdg}(\mathcal Y_{T^o}/T^o)).
\end{eqnarray*}
Consider the exact sequence of variation of mixed Hodge structure over $S$
\begin{eqnarray*}
0\to W_{2p-1}E_{Hdg}^{2p}(\mathcal E/S)\to E_{Hdg}^{2p}(\mathcal E/S)
\xrightarrow{(i_i^*)_{0\leq i\leq r}:=q}\oplus_{i=0}^rE_{Hdg}^{2p}(\mathcal E_i/S)\to 0.  
\end{eqnarray*}
It induces the exact sequence of presheaves on $S_{\mathbb C}^{an}$,
\begin{eqnarray*}
HL^{2p,p}(\mathcal E_{\mathbb C}/S_{\mathbb C})\xrightarrow{q}
\oplus_{i=0}^rHL^{2p,p}(\mathcal E_{i,\mathbb C}/S_{\mathbb C})\xrightarrow{e} \\
J(W_{2p-1}E^{2p}_{Hdg}(\mathcal E_{\mathbb C}/S_{\mathbb C})):=
\Ext^1(\mathbb Q_S^{Hdg},W_{2p-1}E^{2p}_{Hdg}(\mathcal E_{\mathbb C}/S_{\mathbb C}))\xrightarrow{\iota}  
\Ext^1(\mathbb Q_S^{Hdg},E^{2p}_{Hdg}(\mathcal E_{\mathbb C}/S_{\mathbb C}))
\end{eqnarray*}
where 
$HL^{2p,p}(\mathcal E_{\mathbb C}/S_{\mathbb C})\subset E_{DR}^{2p}(\mathcal E_{\mathbb C}/S_{\mathbb C})$, 
$HL^{2p,p}(\mathcal E_{i,\mathbb C}/S_{\mathbb C})\subset E_{DR}^{2p}(\mathcal E_{i,\mathbb C}/S_{\mathbb C})$ 
are the locus of Hodge classes.
Since the monodromy of $R^{2p}\tilde f_{X*}\mathbb Q_{\mathcal Y_{T^o,\mathbb C}}^{an}$ is irreducible
as $i_{T^o*}:\pi_1(T_{\mathbb C}^o)\to\pi_1(T_{\mathbb C})\to\pi_1(\mathbb P(l_X)\backslash\Delta)$ 
is surjective, we have 
\begin{equation}\label{qSp}
Sp_{\mathcal E/\tilde f_X}(\ker q):=Sp_{\mathcal E/\tilde f_X}(W_{2p-1}E_{DR}^{2p}(\mathcal E/S))=0,
\end{equation}
using the fact that there exists a neighborhood 
$V_{\bar S\cup\bar S'}\subset\bar T_{\mathbb C}$ of $\bar S_{\mathbb C}\cup\bar S'_{\mathbb C}$ in $\bar T_{\mathbb C}$ 
for the usual complex topology such that the inclusion  
$i_{\bar S_{\mathbb C}\cup\bar S'_{\mathbb C}}:
\bar S_{\mathbb C}\cup\bar S'_{\mathbb C}\hookrightarrow V_{\bar S\cup\bar S'}$ admits a
retraction $r:V_{\bar S\cup\bar S'}\to\bar S_{\mathbb C}\cup\bar S'_{\mathbb C}$ which is an homotopy equivalence.
On the other hand, since $\tilde f_X\circ i_0:\mathcal E_0\to S$ is a smooth projective morphism,
$E_{Hdg}^{2p}(\mathcal E_0/S)$ is a variation of pure Hodge structure over $S$ polarized by Poincare duality $<-,->$.
In particular, by the proof of Deligne semi-simplicity theorem using Schimdt results, we have a splitting of
variation of pure Hodge structure over $S$
\begin{equation}\label{SqSp}
E_{Hdg}^{2p}(\mathcal E_0/S)=q(\ker Sp_{\mathcal E/\tilde f_X})\oplus q(\ker Sp_{\mathcal E/\tilde f_X})^{\perp,<-,->}, \;
\pi_K:E_{Hdg}^{2p}(\mathcal E_0/S)\to q(\ker Sp_{\mathcal E/\tilde f_X}).
\end{equation}
Let $\lambda\in F^pH^{2p}(E_0^{an},\mathbb Q)$, where we recall $E_0=\mathcal E_{0,\mathbb C,t}$
and $E=\mathcal E_{\mathbb C,t}=\cup_{i=0}^rE_i$. Consider then
$\tilde\lambda\in H^{2p}(E^{an},\mathbb Q)$, such that $q(\tilde\lambda)=\lambda$, and 
\begin{equation*}
Sp_{\mathcal E/\tilde f_X}(\tilde\lambda)\in Sp_{E/f_X}(H^{2p}(E^{an},\mathbb Q))
\subset i_t^*\psi_SR^{2p}\tilde f_{X*}\mathbb Q_{\mathcal Y_{T^o,\mathbb C}^{an}}. 
\end{equation*}
By (\ref{SqSp}), we have
\begin{equation*}
\lambda=\lambda^K+\lambda^L\in F^pH^{2p}(E_0^{an},\mathbb Q), \;
\lambda^K\in i_t^*F^pq(\ker Sp_{\mathcal E/\tilde f_X})_{\mathbb Q}, \,
\lambda^L\in i_t^*F^pq(\ker Sp_{\mathcal E/\tilde f_X})^{\perp}_{\mathbb Q}.
\end{equation*} 
By (\ref{qSp}), if $\lambda\in i_t^*F^pq(\ker Sp_{\mathcal E/\tilde f_X})^{\perp}_{\mathbb Q}$,
the locus of Hodge classes passing through $\lambda$ 
\begin{eqnarray*}
\mathbb V_S^p(\lambda):=\mathbb V_S(\lambda)\cap F^pE^{2p}_{DR}(\mathcal E_{0,\mathbb C}/S_{\mathbb C})
\subset E^{2p}_{DR}(\mathcal E_{0,\mathbb C}/S_{\mathbb C}),
\end{eqnarray*}
inside the De Rham vector bundle of $\tilde f_X\circ i_{E_0}:\mathcal E_0\to S$ satisfy
\begin{eqnarray}\label{qWSeq}
\mathbb V_S^p(\lambda)=
q(Sp_{\mathcal E/\tilde f_X}^{-1}(\mathbb V_S^p(Sp_{\mathcal E/\tilde f_X}(\tilde\lambda))))\cap\pi_K^{-1}(0)
\subset E^{2p}_{DR}(\mathcal E_{0,\mathbb C}/S_{\mathbb C}),
\end{eqnarray}
where 
\begin{itemize}
\item $\mathbb V_S(\lambda)\subset E^{2p}_{DR}(\mathcal E_{0,\mathbb C}/S_{\mathbb C})$,
$\mathbb V_S(Sp_{\mathcal E/\tilde f_X}(\tilde\lambda))\subset E^{2p}_{DR,V_S}(\mathcal Y_{T^o,\mathbb C}/T_{\mathbb C}^o) $,
are the flat leaves, e.g. $\mathbb V_S(\lambda):=\pi_S(\lambda\times\tilde S^{an}_{\mathbb C})$ where
$\pi_S:H^{2p}(E^{an},\mathbb C)\times\tilde S_{\mathbb C}^{an}\to E^{2p,an}_{DR}(\mathcal E_{\mathbb C}/S_{\mathbb C})$ 
is the morphism induced by the universal covering $\pi_S:\tilde S_{\mathbb C}^{an}\to S_{\mathbb C}^{an}$,
\item $q:=q\otimes\mathbb C:E^{2p}_{DR}(\mathcal E_{\mathbb C}/S_{\mathbb C})\to
E^{2p}_{DR}(\mathcal E_{0,\mathbb C}/S_{\mathbb C})$ is the quotient map.
\item $Sp_{\mathcal E/\tilde f_X}:=Sp_{\mathcal E/\tilde f_X}\otimes\mathbb C:
E^{2p}_{DR}(\mathcal E_{\mathbb C}/S_{\mathbb C})\to i_S^{*mod}E^{2p}_{DR,V_S}(\mathcal Y_{T^o,\mathbb C}/T_{\mathbb C}^o)$.
\end{itemize}
Now,
\begin{itemize}
\item If $Sp_{\mathcal E/\tilde f_X}(\tilde\lambda)=0$, we have
by lemma \ref{WSp} applied to $\tilde f_X:\mathcal Y_{\bar T^o}\to\bar T^o$, 
$\tilde\lambda=c_{\mathcal E}(\tilde\lambda)$ with 
$\tilde\lambda\in F^pH^{2p-1}_E(\mathcal Y^{an}_{\mathbb C,C},\mathbb Q)$,
where $C\subset\bar T_{\mathbb C}^o$ is a smooth transversal slice of $S_{\mathbb C}$ at $t$
in particular $\dim(C)=1$, $C\cap S_{\mathbb C}=\left\{t\right\}$ and $\mathcal Y_C$ is smooth, and 
\begin{equation*}
\mathbb V^p_S(\tilde\lambda)=c_{\mathcal E}(\mathbb V^p_S(\tilde\lambda)), \; 
\mathbb V^p_S(\tilde\lambda)\subset 
E_{DR,\mathcal E_{\mathbb C}}^{2p-1}(\mathcal Y_{\bar T^o,\mathbb C}/\bar T_{\mathbb C}^o),
\end{equation*}
with 
\begin{equation*}
c_{\mathcal E}:(E_{DR,\mathcal E_{\mathbb C}}^{2p-1}(\mathcal Y_{\bar T^o,\mathbb C}/\bar T_{\mathbb C}^o),F,W)
\to (E^{2p}_{DR}(\mathcal E_{\mathbb C}/S_{\mathbb C}),F,W).
\end{equation*}
We have for each $t'\in S(\mathbb C)$, the isomorphism
\begin{eqnarray*}
\oplus_{card I=2}F^pH^{2p}(\mathcal E^{an}_{I\mathbb C,t'},\mathbb Q)=
F^p\Gr^W_{2p}H^{2p-2}((\bar{\mathcal Y_{\mathbb C,C'}}\backslash\mathcal E_{t'})^{an},\mathbb Q) \\
\xrightarrow{c(\mathbb Z(\mathcal Y_{\mathbb C,C'}\backslash\mathcal E_{t'}))}
F^p\Gr^W_{2p}H^{2p-1}_E(\bar{\mathcal Y^{an}_{\mathbb C,C'}},\mathbb Q),
\end{eqnarray*}
where $C'\subset\bar T_{\mathbb C}^o$ is a smooth transversal slice of $S$ at $t'$ (hence $\mathcal Y_{C'}$ is smooth) and
$\bar{\mathcal Y_{C'}}$ is a smooth compactification of $\mathcal Y_{C'}$.
Hence, assuming the Hodge conjecture for smooth projective varieties of dimension less or equal to $d-1$,
\begin{equation*} 
\mathbb V^p_S(\lambda)\subset 
\Gr^W_{2p}E_{DR,\mathcal E_{\mathbb C}}^{2p-1}(\mathcal Y_{\bar T^o,\mathbb C}/\bar T_{\mathbb C}^o), \;
\mathbb V^p_S(\lambda)=c_{\mathcal E}(\mathbb V^p_S(\lambda))
\subset E^{2p}_{DR}(\mathcal E_{0\mathbb C}/S_{\mathbb C})
\subset\Gr^W_{2p}E^{2p}_{DR}(\mathcal E_{\mathbb C}/S_{\mathbb C})
\end{equation*}
are defined over $\bar{\mathbb Q}$ and its Galois conjugates are also components of the locus of Hodge classes,
since $\dim(\mathcal E_{I\mathbb C,t'})=d-1$ for $I\subset[0,\cdots,r]$ such that $card I=2$ and $t'\in S(\mathbb C)$. 
\item Consider now the case where $\lambda\in i_t^*F^pq(\ker Sp_{\mathcal E/\tilde f_X})^{\perp}_{\mathbb Q}$.
By \cite{B8} theorem 1, the Hodge conjecture holds for projective hypersurfaces, hence
\begin{equation*}
\mathbb V_S^p(Sp_{\mathcal E/\tilde f_X}(\tilde\lambda))
=\overline{\mathbb V_{T^o}^p(Sp_{\mathcal E/\tilde f_X}(\tilde\lambda))}\cap p^{-1}(S_{\mathbb C})
\subset E^{2p}_{DR,V_S}(\mathcal Y_{T^o,\mathbb C}/T_{\mathbb C}^o),
\end{equation*}
where $p:=p\otimes\mathbb C:E^{2p}_{DR,V_0}(\mathcal Y_{T^o,\mathbb C}/T^o_{\mathbb C})\to\bar T^o_{\mathbb C}$
is the projection, is an algebraic subvariety defined over $\bar{\mathbb Q}$ 
and its Galois conjugates are also components of the locus of Hodge classes,
where the equality follows from \cite{DCK} lemma 2.11. Hence, using (\ref{qWSeq}), 
\begin{equation*}
\mathbb V_S^p(\lambda)=
q(Sp_{\mathcal E/\tilde f_X}^{-1}(\mathbb V_S^p(Sp_{\mathcal E/\tilde f_X}(\tilde\lambda))))\cap\pi_K^{-1}(0)
\subset E^{2p}_{DR}(\mathcal E_{0\mathbb C}/S_{\mathbb C}) 
\end{equation*}
is an algebraic subvariety defined over $\bar{\mathbb Q}$ 
and its Galois conjugates are also components of the locus of Hodge classes. 
\end{itemize}
Hence, 
\begin{itemize}
\item $\mathbb V_S^p(\lambda_K)\subset E^{2p}_{DR}(\mathcal E_{0\mathbb C}/S_{\mathbb C})$ 
is an algebraic subvariety defined over $\bar{\mathbb Q}$ 
and its Galois conjugates are also components of the locus of Hodge classes gives, thus by \cite{B7} theorem 4, 
\begin{equation*}
\lambda^K=[Z^K]\in H^{2p}(E^{an}_0,\mathbb Q), \; Z^K\in\mathcal Z^p(E_0),
\end{equation*}
\item $\mathbb V_S^p(\lambda_S)\subset E^{2p}_{DR}(\mathcal E_{0\mathbb C}/S_{\mathbb C})$ 
is an algebraic subvariety defined over $\bar{\mathbb Q}$ 
and its Galois conjugates are also components of the locus of Hodge classes, thus by \cite{B7} theorem 4, 
\begin{equation*}
\lambda^L=[Z^L]\in H^{2p}(E^{an}_0,\mathbb Q), \; Z^L\in\mathcal Z^p(E_0).
\end{equation*}
We recall the key point is that since 
$\mathbb V_S^p(\lambda^L)\subset E^{2p}_{DR}(\mathcal E_{0\mathbb C}/S_{\mathbb C})$ 
is an algebraic subvariety defined over $\bar{\mathbb Q}$ 
and its Galois conjugate are also components of the locus of Hodge classes, we have
for each $\theta\in Aut(\mathbb C/\mathbb Q)$,
\begin{equation*}
\theta(\lambda^L)\in Gal(\bar{\mathbb Q}/\mathbb Q)(\mathbb V_S^p(\lambda^L)_{\theta})
\subset E^{2p}_{DR}(\mathcal E_{0\mathbb C,\theta}/S_{\mathbb C,\theta}),
\end{equation*}
which implies that 
$\theta(\lambda^L)\in H^{2p}(E^{an}_{0,\theta},\mathbb Q)\subset H^{2p}(E^{an}_{0,\theta},\mathbb C)$,
that is $\lambda^L\in H^{2p}(E^{an}_0,\mathbb Q)$ is absolute Hodge and we apply \cite{B7} corollary 1.
\end{itemize}
Thus,
\begin{equation*}
\lambda=\lambda^K+\lambda^L=[Z]\in H^{2p}(E^{an}_0,\mathbb Q), \; Z:=Z^K+Z^L\in\mathcal Z^p(E_0).
\end{equation*}
Since $p\in\mathbb Z$ and $\lambda\in F^pH^{2p}(E_0^{an},\mathbb Q)$ are arbitrary, 
this proves the Hodge conjecture for $E_0$.
Since $X$ is birational to $E_0$,
the Hodge conjecture for $E_0$ and the induction hypothesis implies the Hodge conjecture for $X$.

I am grateful to professor F.Mokrane for help and support during this work.

\section{Preliminaries and notations}

\subsection{Notations}

\begin{itemize}

\item Denote by $\Top$ the category of topological spaces and $\RTop$ the category of ringed spaces.
\item Denote by $\Cat$ the category of small categories and $\RCat$ the category of ringed topos.
\item For $\mathcal S\in\Cat$ and $X\in\mathcal S$, we denote $\mathcal S/X\in\Cat$ the category whose
objects are $Y/X:=(Y,f)$ with $Y\in\mathcal S$ and $f:Y\to X$ is a morphism in $\mathcal S$, and whose morphisms
$\Hom((Y',f'),(Y,f))$ consists of $g:Y'\to Y$ in $\mathcal S$ such that $f\circ g=f'$.

\item Let $(\mathcal S,O_S)\in\RCat$ a ringed topos with topology $\tau$. For $F\in C_{O_S}(\mathcal S)$,
we denote by $k:F\to E_{\tau}(F)$ the canonical flasque resolution in $C_{O_S}(\mathcal S)$ (see \cite{B5}).
In particular for $X\in\mathcal S$, $H^*(X,E_{\tau}(F))\xrightarrow{\sim}\mathbb H_{\tau}^*(X,F)$.

\item For $f:\mathcal S'\to\mathcal S$ a morphism with $\mathcal S,\mathcal S'\in\RCat$,
endowed with topology $\tau$ and $\tau'$ respectively, we denote for $F\in C_{O_S}(\mathcal S)$ and each $j\in\mathbb Z$,
\begin{itemize}
\item $f^*:=H^j\Gamma(\mathcal S,k\circ\ad(f^*,f_*)(F)):\mathbb H^j(\mathcal S,F)\to\mathbb H^j(\mathcal S',f^*F)$,
\item $f^*:=H^j\Gamma(\mathcal S,k\circ\ad(f^{*mod},f_*)(F)):\mathbb H^j(\mathcal S,F)\to\mathbb H^j(\mathcal S',f^{*mod}F)$, 
\end{itemize}
the canonical maps.

\item For $m:A\to B$, $A,B\in C(\mathcal A)$, $\mathcal A$ an additive category, 
we denote $c(A):\Cone(m:A\to B)\to A[1]$ and $c(B):B\to\Cone(m:A\to B)$ the canonical maps.

\item Denote by $\Sch\subset\RTop$ the subcategory of schemes (the morphisms are the morphisms of locally ringed spaces).
We denote by $\PSch\subset\Sch$ the full subcategory of proper schemes.
For a field $k$, we consider $\Sch/k:=\Sch/\Spec k$ the category of schemes over $\Spec k$.
The objects are $X:=(X,a_X)$ with $X\in\Sch$ and $a_X:X\to\Spec k$ a morphism
and the morphisms are the morphisms of schemes $f:X'\to X$ such that $f\circ a_{X'}=a_X$. We then denote by
\begin{itemize}
\item $\Var(k)=\Sch^{ft}/k\subset\Sch/k$ the full subcategory consisting of algebraic varieties over $k$, 
i.e. schemes of finite type over $k$,
\item $\PVar(k)\subset\QPVar(k)\subset\Var(k)$ 
the full subcategories consisting of quasi-projective varieties and projective varieties respectively, 
\item $\PSmVar(k)\subset\SmVar(k)\subset\Var(k)$,  $\PSmVar(k):=\PVar(k)\cap\SmVar(k)$,
the full subcategories consisting of smooth varieties and smooth projective varieties respectively.
\end{itemize}

\item Denote by $\AnSp(\mathbb C)\subset\RTop$ the subcategory of analytic spaces over $\mathbb C$,
and by $\AnSm(\mathbb C)\subset\AnSp(\mathbb C)$ the full subcategory of smooth analytic spaces 
(i.e. complex analytic manifold).

\item For $X\in\Var(k)$ and $X=\cup_{i\in I}X_i$ with $i_i:X_i\hookrightarrow X$ closed embeddings, 
we denote $X_{\bullet}\in\Fun(\Delta,\Var(k))$ the associated simplicial space,
with for $J\subset I$, $i_{IJ}:X_I:=\cap_{i\in I}X_i\hookrightarrow X_J:=\cap_{i\in J}X_i$ the closed embedding.  

\item Let $(X,O_X)\in\RTop$. We consider its De Rham complex $\Omega_X^{\bullet}:=DR(X)(O_X)$.
\begin{itemize}
\item Let $X\in\Sch$. Considering its De Rham complex $\Omega_X^{\bullet}:=DR(X)(O_X)$,
we have for $j\in\mathbb Z$ its De Rham cohomology $H^j_{DR}(X):=\mathbb H^j(X,\Omega^{\bullet}_X)$.
\item Let $X\in\Var(k)$. Considering its De Rham complex $\Omega_X^{\bullet}:=\Omega_{X/k}^{\bullet}:=DR(X/k)(O_X)$,
we have for $j\in\mathbb Z$ its De Rham cohomology $H^j_{DR}(X):=\mathbb H^j(X,\Omega^{\bullet}_X)$.
The differentials of $\Omega_X^{\bullet}:=\Omega_{X/k}^{\bullet}$ are by definition $k$-linear,
thus $H^j_{DR}(X):=\mathbb H^j(X,\Omega^{\bullet}_X)$ has a structure of a $k$ vector space.
\item Let $X\in\AnSp(\mathbb C)$. Considering its De Rham complex $\Omega_X^{\bullet}:=DR(X)(O_X)$,
we have for $j\in\mathbb Z$ its De Rham cohomology $H^j_{DR}(X):=\mathbb H^j(X,\Omega^{\bullet}_X)$.
\end{itemize}

\item For $Y\in\Var(\mathbb Q)$ and $X=V(I)=V(f_1,\ldots,f_r)\subset Y_{\mathbb C}$, we consider
$\mathcal X=V(\tilde I)=V(\tilde f_1,\ldots,\tilde f_r)\subset Y\times\mathbb A_{\mathbb Q}^s$ the
canonical rational deformation, where  
\begin{eqnarray*}
\mbox{if} \, f=\sum_{I\subset[1,\ldots,d]^n} a_Ix_1^{n_1}\cdots x_n^{n_n}\in\mathbb C[x_1,\ldots,x_n], \\
\tilde f=\sum_{I\subset[1,\ldots,d]^n} a_Ix_1^{n_1}\cdots x_n^{n_n}
\in\mathbb Q[(a_I)_{I\subset[1,\ldots,d]^n},x_1,\ldots,x_n].
\end{eqnarray*}

\item For $X\in\AnSp(\mathbb C)$, we denote $\alpha(X):\mathbb C_X\hookrightarrow\Omega_X^{\bullet}$ the embedding in $C(X)$.
For $X\in\AnSm(\mathbb C)$, $\alpha(X):\mathbb C_X\hookrightarrow\Omega_X^{\bullet}$ 
is an equivalence usu local by Poincare lemma.

\item We denote $\mathbb I^n:=[0,1]^n\in\Diff(\mathbb R)$ (with boundary).
For $X\in\Top$ and $R$ a ring, we consider its singular cochain complex
\begin{equation*}
C^*_{\sing}(X,R):=(\mathbb Z\Hom_{\Top}(\mathbb I^*,X)^{\vee})\otimes R 
\end{equation*}
and for $l\in\mathbb Z$ its singular cohomology $H^l_{\sing}(X,R):=H^nC^*_{\sing}(X,R)$.
For $f:X'\to X$ a continuous map with $X,X'\in\Top$, we have the canonical map of complexes
\begin{equation*}
f^*:C^*_{\sing}(X,R)\to C^*_{\sing}(X,R), \sigma\mapsto f^*\sigma:=(\gamma\mapsto\sigma(f\circ\gamma)).
\end{equation*}
In particular, we get by functoriality the complex 
\begin{equation*}
C^*_{X,R\sing}\in C_R(X), \; (U\subset X)\mapsto C^*_{\sing}(U,R)
\end{equation*}
We recall that 
\begin{itemize}
\item For $X\in\Top$ locally contractible, e.g. $X\in\CW$, and $R$ a ring, the inclusion in $C_R(X)$
$c_X:R_X\to C^*_{X,R\sing}$ is by definition an equivalence top local and that we get 
by the small chain theorem, for all $l\in\mathbb Z$, an isomorphism 
$H^lc_X:H^l(X,R_X)\xrightarrow{\sim}H^l_{\sing}(X,R)$.
\item For $X\in\Diff(\mathbb R)$, the restriction map 
\begin{equation*}
r_X:\mathbb Z\Hom_{\Diff(\mathbb R)}(\mathbb I^*,X)^{\vee}\to C^*_{\sing}(X,R), \; 
w\mapsto w:(\phi\mapsto w(\phi))
\end{equation*}
is a quasi-isomorphism by Whitney approximation theorem.
\end{itemize}

\item Let $S\in\AnSm(\mathbb C)$ and $L\in\Shv(S)$ a local system. 
Consider $E:=L\otimes O_S\in\Vect_{\mathcal D}(S)$ 
the corresponding holomorphic vector bundle with integrable connection $\nabla$. 
Let $\pi_S:\tilde S\to S$ be the universal covering. Consider the canonical fiber 
\begin{equation*}
L_0:=\Gamma(\tilde S,\pi_S^*L)=\Gamma(\tilde S,\pi_S^{*mod}E)^{\nabla}.
\end{equation*}
For $\lambda\in L_0$, we will consider the flat leaf 
\begin{equation*}
\mathbb V_S(\lambda):=\pi_S(\lambda\times\tilde S)\subset E, \; \pi_S:L_0\times\tilde S\to E, \,
\pi_S(\nu,z):=(\nu(z),\pi_S(z)).
\end{equation*}
Note that for $t\in S$, $\mathbb V_S(\lambda)\cap p_S^{-1}(t)=\pi_1(S,t)(\lambda)$ is the orbit
of $\lambda$ under the monodromy action, where $p_S:E\to S$ is the projection.

\end{itemize}

\subsection{Birational projective varieties and Hodge conjecture}

In this subsection, we recall that assuming the Hodge conjecture 
for smooth projective varieties of dimension less or equal to $d-1$, 
the property of satisfying the Hodge conjecture for connected smooth projective varieties of dimension $d$
is a birational invariant :

\begin{prop}\label{birHdg}
\begin{itemize}
\item[(i)] Let $X\in\PSmVar(\mathbb C)$ connected of dimension $d$.
Let $\epsilon:\tilde X_Z\to X$ be the blow up of $X$ along a smooth closed subvariety $Z\subset X$. 
Assume the Hodge conjecture hold for smooth complex projective varieties of dimension less or equal to $d-1$.
Then the Hodge conjecture hold for $X$ if and only if it hold for $\tilde X_Z$.
\item[(ii)]Let $X,X'\in\PSmVar(\mathbb C)$ connected of dimension $d$. 
Assume the Hodge conjecture hold for smooth complex projective varieties of dimension less or equal to $d-1$.
If $X$ is birational to $X'$, then the Hodge conjecture hold for $X$ if and only if it hold for $X'$
\end{itemize}
\end{prop}

\begin{proof}
\noindent(i): Follows from the fact that 
\begin{equation*}
(\epsilon^*,\oplus_{l=1}^ci_{E*}((-).h^{l-1})):
H^k(X,\mathbb Q)\oplus\oplus_{l=1}^c H^{k-2l}(E,\mathbb Q)\to H^k(\tilde X_Z,\mathbb Q)
\end{equation*}
is an isomorphism of Hodge structures for each $k\in\mathbb Z$, where $c=\codim(Z,X)$,
$i_E:E\hookrightarrow\tilde X_Z$ is the closed embedding, $\epsilon_{|E}:E\to Z$ being a projective vector bundle.

\noindent(ii):Follows from (i) since if $\pi:X^o\to X'$ is a birational map, $X^o\subset X$ being an open subset, 
then $X$ is connected to $X'$ by a sequence of blow up of smooth projective varieties with smooth center by \cite{BirBl}.
\end{proof}

\subsection{A family of smooth projective hypersurfaces associated to a smooth projective variety}

In this subsection, we recall that given a smooth complex projective variety $X$, there exists
a family of smooth complex projective hypersurfaces which degenerates into a normal crossing divisor with
one irreducible component birational to $X$.

\begin{prop}\label{XEY}
Let $X\in\PSmVar(\mathbb C)$ connected of dimension $d$.
Then there exists a family of smooth projective hypersurfaces $f_X:Y\to\mathbb A^1_{\mathbb C}$ with
\begin{itemize}
\item $Y\in\SmVar(\mathbb C)$, $f_X$ flat projective, 
\item $Y_s:=f_X^{-1}(s)\subset\mathbb P^{d+1}_{\mathbb C}$ for $s\in\mathbb A^1_{\mathbb C}\backslash 0$ 
are smooth projective hypersurfaces,
\item $E:=f_X^{-1}(0)=\cup_{i=0}^rE_i=\epsilon^{-1}(Y_0)\subset Y$ is a normal crossing divisor,
where $Y_0\subset\mathbb P^{d+1}_{\mathbb C}$ is a possibly singular hypersurface birationnal to $X$,
$\epsilon:(Y,E)\to (Y',Y_0)$ a desingularization, and $E_0$ is birational to $X$,
\end{itemize}
\end{prop}

\begin{proof}
See \cite{Ayoub} lemma 1.5.8 . 
It use the fact that $X$ is birational to a possibly singular projective hypersurface $Y_0$ and
consider the desingularization of a generic pencil 
$f'_X:Y'=V(f)\subset\mathbb P^{d+1}_{\mathbb C}\times\mathbb A^1_{\mathbb C}\to\mathbb A^1_{\mathbb C}$ 
passing through $Y_0$ and $f_X=f'_X\circ\epsilon$.
\end{proof}

\subsection{The nearby cycle functor and the specialization map}

Recall from \cite{B5} that for $k\subset\mathbb C$ a subfield and $S\in\Var(k)$ quasi-projective, 
and $l:S\hookrightarrow\tilde S$ a closed embedding with $\tilde S\in\SmVar(k)$,
we have the full subcategory 
\begin{equation*}
\iota:MHM_{k,gm}(S)\hookrightarrow\PSh_{\mathcal Dfil,S}(\tilde S)\times_IP(S^{an}_{\mathbb C})
\end{equation*}
consisting of geometric mixed Hodge module whose De Rham part is defined over $k$, where,
using the fact that the $V$-filtration is defined over $k$,
$\PSh_{\mathcal D(1,0)fil,S}(\tilde S)\times_IP_{fil}(S^{an}_{\mathbb C})$ is the category 
\begin{itemize}
\item whose objects are 
$((M,F,W),(K,W),\alpha)$ where $(M,F,W)$ is a filtered $D_{\tilde S}$ module supported on $S$, where 
$l:S\hookrightarrow\tilde S$ is a closed embedding and $\tilde S\in\SmVar(k)$,
$(K,W)$ is a filtered Perverse sheaf on $S^{an}_{\mathbb C}$, and 
$\alpha:l_*(K,W)\otimes\mathbb C\to DR(\tilde S)((M,W)^{an})$ is an isomorphism in $D_{fil}(\tilde S^{an}_{\mathbb C})$.
\item whose morphism between $((M,F,W),(K,W),\alpha)$ and $((M',F,W),(K',W),\alpha')$, are 
couples of morphisms $(\phi:(M,F,W)\to (M',F,W),\psi:(K,W)\to(K',W))$ such that $\phi\circ\alpha=\alpha'\circ\psi$.
\end{itemize}
Then, by \cite{B5}, we get for $S\in\Var(k)$ quasi-projective 
and $l:S\hookrightarrow\tilde S$ is a closed embedding with $\tilde S\in\SmVar(k)$ an embedding
\begin{equation*}
\iota:D(MHM_{k,gm}(S))\hookrightarrow D_{\mathcal Dfil,S}(\tilde S)\times_ID(S^{an}_{\mathbb C})
\end{equation*}
consisting of mixed Hodge module whose De Rham part is defined over $k$, where
$D_{\mathcal D(1,0)fil,S}(\tilde S)\times_ID_{fil,c}(S^{an}_{\mathbb C})$ is the category 
\begin{itemize}
\item whose objects are 
$((M,F,W),(K,W),\alpha)$ where $(M,F,W)$ is a complex of filtered $D_{\tilde S}$ module supported on $S$, 
$(K,W)$ is a filtered complex of presheaves on $S^{an}_{\mathbb C}$ whose cohomology are constructible sheaves, and 
$\alpha:l_*(K,W)\otimes\mathbb C\to DR(\tilde S)((M,W)^{an})$ is an isomorphism in $D_{fil}(\tilde S^{an}_{\mathbb C})$.
\item whose morphism are given in \cite{B5},
\end{itemize}
We have then the six functors formalism : 
\begin{eqnarray*}
D(MHM_{k,gm}(-)):\Var(k)\mapsto\TriCat, \; S\mapsto D(MHM_{k,gm}(S)), \; (f:T\to S)\mapsto \\
(f^*,f_*):D(MHM_{k,gm}(S))\to D(MHM_{k,gm}(T)), \; (f_!,f^!):D(MHM_{k,gm}(T))\to D(MHM_{k,gm}(S))
\end{eqnarray*}
and for $D\subset S$ a (Cartier) divisor the nearby cycle functor
\begin{eqnarray*}
\psi_D:D(MHM_{k,gm}(S))\to D(MHM_{k,gm}(D)), \\
\psi_D((M,F,W),(K,W),\alpha):=(\psi_D(M,F,W),\psi_D(K,W),\psi_D(\alpha)).
\end{eqnarray*}
For $T\in\Var(k)$ and $\epsilon_{\bullet}:T_{\bullet}\to T$ a desingularization with $T_{\bullet}\in\Fun(\Delta,\SmVar(k))$,
we have
\begin{equation*}
\mathbb Z_T^{Hdg}:=a_T^{*Hdg}\mathbb Z^{Hdg}=
(\epsilon_{\bullet*}E_{zar}(O_{T_{\bullet}},F,W),\epsilon_{\bullet*}E_{usu}(\mathbb Z_{T^{an}_{\bullet}},W),\alpha(T_{\bullet}))
\end{equation*}
For $f:X\to S$ a projective morphism with $X,S\in\SmVar(k)$, we denote 
\begin{equation*}
E_{Hdg}(X/S):=f_*\mathbb Z_X^{Hdg}:=(\int_f(O_X,F_b),Rf_*\mathbb Z_{X^{an}_{\mathbb C}},f_*\alpha(X))
\in D(MHM_{k,gm}(S))
\end{equation*}
and $E^j_{Hdg}(X/S):=H^jE_{Hdg}(X/S)$ for $j\in\mathbb Z$. If moreover $Z\subset X$ is a closed subset, we denote 
\begin{equation*}
E_{Hdg,Z}(X/S):=f_*\Gamma_Z^{Hdg}\mathbb Z_X^{Hdg}:=
(\int_f\Gamma_Z^{Hdg}(O_X,F_b),Rf_*R\Gamma_Z\mathbb Z_{X^{an}_{\mathbb C}},f_*\Gamma_Z\alpha(X))\in D(MHM_{k,gm}(S))
\end{equation*}
and $E^j_{Hdg,Z}(X/S):=H^jE_{Hdg,Z}(X/S)$ for $j\in\mathbb Z$.
For $f:X\hookrightarrow\mathbb P^N\times S\xrightarrow{p}S$ a projective morphism with $X,S\in\Var(k)$, 
$S$ smooth, we denote 
\begin{equation*}
E_{Hdg}(X/S):=f_*\mathbb Z_X^{Hdg}:=
(\int_p\Gamma^{\vee,Hdg}_X(O_{\mathbb P^N\times S},F),Rf_*\mathbb Z_{X^{an}_{\mathbb C}},f_*\alpha(X))
\in D(MHM_{k,gm}(S))
\end{equation*}
and $E^j_{Hdg}(X/S):=H^jE_{Hdg}(X/S)$ for $j\in\mathbb Z$.
In the particular case where $E_{Hdg}^j(X/S)$ is a variation of mixed Hodge structure, the first factor of
\begin{equation*}
E^j_{Hdg}(X/S):=(E^j_{DR}(X/S),R^jf_*\mathbb Z_{X^{an}_{\mathbb C}},H^jf_*\alpha(X))
\in MHM_{k,gm}(S)
\end{equation*}
is a vector bundle.

We then have the specialization map :

\begin{defiprop}\label{SpT}
Let $k\subset\mathbb C$ be a subfield.
\begin{itemize}
\item[(i)]Let $f':X'\to T$ be a projective morphism with $X',T\in\SmVar(k)$. Assume that 
\begin{itemize}
\item $f^o:=f\otimes_TT:X^o\to T^o$ is smooth, $T^o\subset T$ is an open subset, 
\item $X'_S:=f^{-1}(S)\subset X'$, $S:=T\backslash T^o$. We have by definition $X^o=X'\backslash X'_S$.
\end{itemize} 
We have then, see \cite{B5}, the distinguished triangles in $D(MHM_{k,gm}(X_S))$,
\begin{eqnarray*}
\mathbb Z_{X'_S}^{Hdg}:=i_{X'_S}^*\mathbb Z_{X'}^{Hdg}
\xrightarrow{Sp_{X'_S/X'}}\psi_{X'_S}\mathbb Z^{Hdg}_{X^o}:=
(\psi_{X'_S}(O_{X'},F),\psi_{X'_S}\mathbb Z_{X^o},\psi_{X'_S}\alpha(X')) \\
\xrightarrow{c(\psi_{X'_S}\mathbb Z^{Hdg}_{X^o})} 
\phi_{X'_S}\mathbb Z^{Hdg}_{X^o}:=\Cone(Sp_{X'/X_S})\xrightarrow{c(\mathbb Z_{X'_S}^{Hdg})}\mathbb Z_{X'_S}^{Hdg}[1].
\end{eqnarray*}
Applying the functor $(f'\circ i_{X'_S})_*$, we obtain the generalized distinguished triangle in $D(MHM_{k,gm}(S))$
\begin{eqnarray*}
E_{Hdg}(X'_S/S)\xrightarrow{sp_{X'_S/f}:=(f\circ i_{X'_S})_*Sp_{X'_S/X'}} \\
(f'\circ i_{X'_S})_*(\psi_{X'_S}\mathbb Z^{Hdg}_{X^o})=\psi_0E_{Hdg}(X^o/T^o) 
\xrightarrow{(f'\circ i_{X'_S})_*c(\Phi_{X_0}\mathbb Z^{Hdg}_{X^o})} \\
(f'\circ i_{X'_S})_*(\phi_{X_S}\mathbb Z^{Hdg}_{X^o})=\phi_SE_{Hdg}(X^o/T^o)
\xrightarrow{(f'\circ i_{X'_S})_*c(\mathbb Z^{Hdg}_{X'})} E_{Hdg}(X'_S/S)[1]
\end{eqnarray*}
\item[(ii)]Let $f:X\to T$ be a projective morphism with $X,T\in\SmVar(k)$. Assume that 
\begin{itemize}
\item $f^o:=f\otimes_{T}T^o:X^o\to T^o$ is smooth, $T^o\subset T$ an open subset 
\item $S:=T\backslash T^o$, $E:=f^{-1}(S)=\cup_{i=0}^rE_i\subset X$ is a normal crossing divisor and that 
for $I\subset[0,\ldots r]$, $f_{0,I}:=f_{|E_I}:E_I:=\cap_{i\in I}E_i\to S$ are smooth. 
We have by definition $X^o=X\backslash E$.
\end{itemize}
Denote for $I\subset J\subset[0,\ldots r]$, 
$i_{JI}:E_J\hookrightarrow E_I$ and $i_I:E_I\hookrightarrow E$, $i_E:E\hookrightarrow X$ the closed embeddings. 
In this particular case, c.f. \cite{PS}, the distinguished triangle in $D(MHM_{k,gm}(E))$ given in (i) become
\begin{eqnarray*}
\mathbb Z_E^{Hdg}:=i_E^*\mathbb Z_X^{Hdg}=
((i_{\bullet*}\Omega^{\bullet}_{E^{\bullet}},F,W),(\mathbb Z_{E_{\bullet}},W),\alpha(E_{\bullet}))
\xrightarrow{Sp_{E/X}} \\
\psi_E\mathbb Z^{Hdg}_{X^o}=
((i_E^{*mod}\Omega^{\bullet}_{X/T}(\log E),F,W),\psi_E\mathbb Z_{X^o},\psi_E\alpha(X))
\xrightarrow{c(\psi_E\mathbb Z^{Hdg}_{X^o})} \\
\phi_E\mathbb Z^{Hdg}_{X^o}:=\Cone(Sp_{E/X})\xrightarrow{c(\mathbb Z_E^{Hdg})}\mathbb Z_E^{Hdg}[1].
\end{eqnarray*}
Recall the map $Sp_{E/X}$ is the defined as the factorization 
\begin{equation*}
\mathbb Z_X^{Hdg}\xrightarrow{\ad(i_E^*,i_{E*})(\mathbb Z_X^{Hdg})}\mathbb Z_E^{Hdg}\xrightarrow{Sp_{E/X}}
\psi_E\mathbb Z^{Hdg}_{X^o}.
\end{equation*}

Applying the functor $(f\circ i_E)_*$, we obtain the generalized distinguished triangle in $D(MHM_{k,gm}(S))$
\begin{eqnarray*}
E_{Hdg}(E/S)\xrightarrow{Sp_{E/f}:=(f\circ i_E)_*Sp_{E/X}} 
(f\circ i_E)_*(\psi_E\mathbb Z^{Hdg}_{X^o})=\psi_SE_{Hdg}(X^o/T^o) \\
\xrightarrow{(f\circ i_E)_*c(\psi_E\mathbb Z^{Hdg}_{X^o})}
(f\circ i_E)_*(\phi_E\mathbb Z^{Hdg}_{X^o})=\phi_SE_{Hdg}(X^o/T^o)
\xrightarrow{(f\circ i_E)_*c(\mathbb Z^{Hdg}_E)} E_{Hdg}(E/S)[1].
\end{eqnarray*}
If $f=f'\circ\epsilon$ where $\epsilon:(X,E)\to(X,X_S)$ is a desingularization, 
we have $Sp_{X_S/f'}=Sp_{E/f}\circ\epsilon^*$.
\end{itemize}
\end{defiprop}

\begin{proof}
See \cite{PS}.
\end{proof}

\begin{lem}\label{WSp}
Let $k\subset\mathbb C$ be a subfield. 
Let $f:X\to T$ be a projective morphism with $X,T\in\SmVar(k)$. Assume that 
\begin{itemize}
\item $f^o:=f\otimes_{T}T^o:X^o\to T^o$ is smooth of relative dimension $n$, $T^o\subset T$ an open subset 
\item $S:=T\backslash T^o=V(s)\subset T$ is a divisor, 
$E:=f^{-1}(S)=\cup_{i=0}^rE_i\subset X$ is a normal crossing divisor and that 
for $I\subset[0,\ldots r]$, $f_{0,I}:=f_{|E_I}:E_I:=\cap_{i\in I}E_i\to S$ are smooth. 
We have by definition $X^o=X\backslash E$.
\end{itemize}
Denote $j_0:E^o_0:=E_0\backslash\cup_{i=1}^rE_i\xrightarrow{j'_0}E_0\xrightarrow{i_0}E$ the embedding. 
\begin{itemize}
\item[(i)] We have, for each $k\in\mathbb Z$, a canonical map in $MHM_{k,gm}(S)$
\begin{equation*}
c_E:E^{k-1}_{Hdg,E}(X/T)\xrightarrow{c_V}\phi_SE^{k-1}_{Hdg}(X^o/T^o)
\xrightarrow{H^{k-1}(f\circ i_E)_*c(\mathbb Q^{Hdg}_E)}E^k_{Hdg}(E/S).
\end{equation*}
\item[(ii)]If $H^{2p-1}(X_t,\mathbb Q)=H^{2p-2}(X_t,\mathbb Q)=0$ for $t\in T\backslash S$, then 
\begin{eqnarray*}
c_E:E^{k-1}_{Hdg,E}(X/T)\xrightarrow{c_V}\phi_SE^{k-1}_{Hdg}(X^o/T^o) \\
\xrightarrow{H^{k-1}(f\circ i_E)_*c(\mathbb Q^{Hdg}_E)}\ker(sp_{E/f}:E^k_{Hdg}(E/S)\to\psi_SE^k_{Hdg}(X^o/T^o)).
\end{eqnarray*}
is an isomorphism.
\end{itemize}
\end{lem}

\begin{proof}
\noindent(i):See \cite{PS} for the definition of the map $c_V$. Recall that $Var=DR(S)(s)$.

\noindent(ii):Follows from the exact sequences in $MHM_{k,gm}(S)$
\begin{eqnarray*}
\cdots\to\psi_SE_{Hdg}^{k-2}(X^o/T^o)\to E_{Hdg,E}^{k-1}(X/T)\xrightarrow{c_V}\phi_SE_{Hdg}^{k-1}(X^o/T^o)
\xrightarrow{V}\psi_SE_{Hdg}^{k-1}(X^o/T^o)\to\cdots, \\
\cdots\to\psi_SE_{Hdg}^{k-1}(X^o/T^o)\xrightarrow{H^{k-1}(f\circ i_E)_*c(\psi_E\mathbb Q^{Hdg}_{X^o})} \\
\phi_SE_{Hdg}^{k-1}(X^o/T^o)\xrightarrow{H^{k-1}(f\circ i_E)_*c(\mathbb Q^{Hdg}_E)}E_{Hdg}^k(E/S)
\xrightarrow{Sp_{E/f}}\psi_SE_{Hdg}^k(X^o/T^o)\to\cdots,
\end{eqnarray*}
given in \cite{PS} for the first one.
\end{proof}

\section{Hodge conjecture for smooth projective varieties}

\begin{thm}
Let $X\in\PSmVar(\mathbb C)$. Then Hodge conjecture hold for $X$. 
That is, if $p\in\mathbb Z$ and $\lambda\in F^pH^{2p}(X^{an},\mathbb Q)$, $\lambda=[Z]$ with $Z\in\mathcal Z^p(X)$. 
\end{thm}

\begin{proof}
The Hodge conjecture is true for curves and surfaces.
Assume that the Hodge conjecture is true for smooth projective varieties of dimension less or equal to $d-1$.
Let $X\in\PSmVar(\mathbb C)$ of dimension $d$.
Up to split $X$ into its connected components, we may assume that $X$ is connected of dimension $d$.
Then $X$ is birational to a possibly singular hypersurface $Y_{0_X}\subset\mathbb P^{d+1}_{\mathbb C}$
of degree $l_X:=deg(\mathbb C(X)/\mathbb C(x_1,\cdots,x_d))$. Consider
\begin{equation*}
f:\mathcal Y'_{\mathbb C}=V(\tilde f)\subset\mathbb P^{d+1}_{\mathbb C}\times\mathbb P(l_X)\to\mathbb P(l_X)
\end{equation*}
the universal family of projective hypersurfaces of degree $l_X$, 
and $0_X\in\mathbb P(l_X)$ the point corresponding to $Y_{0_X}$. 
Denote $\Delta\subset\mathbb P(l_X)$ the discriminant locus parametrizing the hypersurfaces which are singular.
Let $\bar S:=\bar{0_X}^{\bar{\mathbb Q}}\subset\mathbb P(l_X)$ 
be the $\bar{\mathbb Q}$-Zariski closure of $0_X$ inside $\mathbb P(l_X)$.
If $Y_{0_X}$ is smooth, then $X$ satisfy the hodge conjecture since $Y_{0_X}$ satisfy the Hodge conjecture (\cite{B8})
and by proposition \ref{birHdg}(ii) and the induction hypothesis since $X$ is birational to $Y_{0_X}$.
Hence, we may assume that $Y_{0_X}$ is singular, that is $0_X\in\Delta$. 
Then, $\bar S\subset\Delta$ since $\Delta\subset\mathbb P(l_X)$ is defined over $\bar{\mathbb Q}$.
Let $\bar T\subset\mathbb P(l_X)$ be an irreducible closed subvariety of dimension $\dim(\bar S)+1$ 
defined over $\bar{\mathbb Q}$ such that 
\begin{equation*}
\bar S\subset\bar T \; \mbox{and} \; T:=\bar T\cap(\mathbb P(l_X)\backslash\Delta)\neq\emptyset.
\end{equation*}
We take for $\bar T=V(f_1,\cdots,f_s)\subset\mathbb P(l_X)$ 
an irreducible complete intersection defined over $\bar{\mathbb Q}$ which intersect $\Delta$ properly and such that 
\begin{itemize}
\item $V(f_1)\supset\sing(\Delta)\cup\bar S$ 
\item for each $r\in[1,\cdots,s]$, $V(f_1,\cdots,f_r)\supset\sing(V(f_1,\cdots,f_{r-1})\cap\Delta)\cup\bar S$ 
contain inductively the singular locus of the intersection with $\Delta$ and $\bar S$.
\end{itemize}
Then by the weak Lefchetz hyperplane theorem for homotopy groups,
$T\subset\mathbb P(l_X)\backslash\Delta$ is not contained in a weakly special subvariety since 
\begin{equation*}
i_{T*}:\pi_1(T_{\mathbb C})\to\pi_1(\mathbb P(l_X)\backslash\Delta)
\end{equation*}
is an isomorphism if $\dim\bar T\geq 2$ and surjective if $\dim\bar T=1$.
We have then $\bar T\cap\Delta=\bar S\cup\bar S'$, where $\bar S'\subset\bar T$ is a divisor.
We then consider the family of algebraic varieties defined over $\mathbb Q$
\begin{equation*}
\tilde f_X:\mathcal Y_{\bar T^o}\xrightarrow{\epsilon}\mathcal Y'_{\bar T^o}
\xrightarrow{f\times_{\mathbb P(l_X)}\bar T^o}\bar T^o 
\end{equation*}
where 
\begin{itemize}
\item $\tilde f_X:\mathcal Y\xrightarrow{\epsilon}\mathcal Y'_{\bar T}\xrightarrow{f\times_{\mathbb P(l_X)}\bar T}\bar T$ 
with $\epsilon:(\mathcal Y,\mathcal E')\to(\mathcal Y'_{\bar T},\mathcal Y'_{\bar S}\cup\mathcal Y'_{\bar S'})$ 
is a desingularization over $\mathbb Q$, in particular $\mathcal E'=\mathcal E\cup\mathcal E''$
with $\mathcal E=\cup_{i=0}^r\mathcal E_i:=\tilde f_X^{-1}(\bar S)\subset\mathcal Y$
which is a normal crossing divisor, in particular $\mathcal Y_T=\mathcal Y'_T$,
\item $S\subset\bar S$ is the open subset such that $S$ is smooth and 
$\tilde f_{X}\times_{\bar T}S:\mathcal E_I:=\cap_{i\in I}\mathcal E_i\to S$ is 
smooth projective for each $I\subset[0,\ldots,r]$,
$\bar T^o\subset\bar T$ is an open subset such that $\bar T^o\cap\Delta=S$ and $T^o:=\bar T^o\cap T$ is smooth,
for simplicity we denote again $\mathcal E$ for the open subset $\mathcal E_S:=\tilde f_X^{-1}(S)\subset\mathcal E$,
\item $t\in S(\mathbb C)$ is the point such that $\mathcal Y'_{\mathbb C,t}=Y_{0_X}$ and 
$\epsilon_t:E_0:=\mathcal E_{0,\mathbb C,t}\to Y_{0_X}$ is surjective 
(there is at least one component of $E$ dominant over $Y_{0_X}$ since $E$ is dominant over $Y_{0_X}$
as $\mathcal E$ is dominant over $\mathcal Y'_{\bar S}$), 
in particular $E_0$ is birational to $Y_{0_X}$,hence $E_0$ is birational to $X$. 
We denote $E_i:=\mathcal E_{i,\mathbb C,t}$ for each $0\leq i\leq r$ and $E:=\mathcal E_{\mathbb C,t}=\cup_{i=0}^rE_i$.
\end{itemize}
We then prove, assuming the Hodge conjecture for smooth projective varieties of dimension less or equal to $d-1$, 
the Hodge conjecture for the $E_i$, more specifically for $E_0$.
Denote for $I\subset[0,\ldots,s]$, $i_I:\mathcal E_I\hookrightarrow\mathcal E$ and
for $I\subset J\subset[0,\ldots,s]$, $i_{JI}:\mathcal E_J\to\mathcal E_I$ the closed embeddings.
Then, $i_{\bullet}:\mathcal E_{\bullet}\to\mathcal E$ in $\Fun(\Delta,\Var(\mathbb Q))$ is a simplicial resolution
with $\mathcal E_{\bullet}\in\Fun(\Delta,\SmVar(\mathbb Q))$. We have then
\begin{eqnarray*}
E_{Hdg}(\mathcal E/S):=
(\tilde f_X\circ i_{\bullet*}E_{zar}(\Omega^{\bullet}_{\mathcal E_{\bullet}/S},F_b,W),
(\tilde f_X\circ i_{\bullet*}E_{usu}(\mathbb Q_{\mathcal E_{\bullet,\mathbb C}^{an}},W)),
\tilde f\circ i_{\bullet*}\alpha(\mathcal E_{\bullet}))\in D(MHM_{\mathbb Q,gm}(S))
\end{eqnarray*}
Then, definition-proposition \ref{SpT} applied to $f_T=f\times_{\mathbb P(l_X)}\bar T^o:\mathcal Y'_{\bar T^o}\to\bar T^o$ 
gives the distinguish in $D(MHM_{\mathbb Q,gm}(S))$, 
\begin{eqnarray*}
E_{Hdg}(\mathcal Y'_S/S)
\xrightarrow{Sp_{\mathcal Y'_S/f_T}:=(f_T\circ i_{\mathcal Y'_S})_*Sp_{\mathcal Y_S/\mathcal Y'_{\bar T^o}}} 
\psi_SE_{Hdg}(\mathcal Y'_{T^o}/T^o) \\
\xrightarrow{(f_T\circ i_{\mathcal Y'_S})_*c(\psi_{\mathcal Y'_S}\mathbb Z^{Hdg}_{\mathcal Y_{T^o}})}
\phi_SE_{Hdg}(\mathcal Y'_{T^o}/T^o)
\xrightarrow{(f_T\circ i_{\mathcal Y'_S})_*c(\mathbb Z^{Hdg}_{\mathcal Y'_S)}} E_{Hdg}(\mathcal Y'_S/S)[1]
\end{eqnarray*}
and definition-proposition \ref{SpT} applied to $\tilde f_X:\mathcal Y_{\bar T^o}\to\bar T^o$ 
gives the distinguish in $D(MHM_{\mathbb Q,gm}(S))$, 
\begin{eqnarray*}
E_{Hdg}(\mathcal E/S)
\xrightarrow{Sp_{\mathcal E/\tilde f_X}:=(\tilde f_X\circ i_{\mathcal E})_*Sp_{\mathcal E/\mathcal Y}} 
\psi_SE_{Hdg}(\mathcal Y_{T^o}/T^o) \\
\xrightarrow{(\tilde f_X\circ i_{\mathcal E})c(\psi_{\mathcal E}\mathbb Z^{Hdg}_{\mathcal Y_{T^o}})}
\phi_SE_{Hdg}(\mathcal Y_{T^o}/T^o)
\xrightarrow{(\tilde f_X\circ i_{\mathcal E})_*c(\mathbb Z^{Hdg}_{\mathcal E)}} E_{Hdg}(\mathcal E/S)[1]
\end{eqnarray*}
and $Sp_{\mathcal E/\tilde f_X}\circ\epsilon^*=Sp_{\mathcal Y_{\bar T^o}/f_T}$. Let $p\in\mathbb Z$. 
By the hard Lefchetz theorem on $H^*(E_0)$, we may assume that $2p\leq d$.
Taking the cohomology to this distinguish triangle, we get in particular, for each $p\in\mathbb Z$,
the maps in $MHM_{\mathbb Q,gm}(S)$
\begin{itemize} 
\item $Sp_{\mathcal Y'_S/f_T}:E^{2p}_{Hdg}(\mathcal Y'_S/S)\to H^{2p}\psi_SE_{Hdg}(\mathcal Y_{T^o}/T^o)$ 
\item $Sp_{\mathcal E/\tilde f_X}:E^{2p}_{Hdg}(\mathcal E/S)\to H^{2p}\psi_SE_{Hdg}(\mathcal Y_{T^o}/T^o)$.
\end{itemize}
Moreover, we have, see lemma \ref{WSp}, an isomorphism of variation of mixed Hodge structures over $S$
\begin{eqnarray*}
c_{\mathcal E}:=H^{2p-1}(\tilde f_X\circ i_{\mathcal E})_*c(\mathbb Q^{Hdg}_{\mathcal E})\circ c_V:
E_{Hdg,\mathcal E}^{2p-1}(\mathcal Y_{\bar T^o}/\bar T^o)\xrightarrow{\sim} \\
\phi_SE_{Hdg,\mathcal E}^{2p-1}(\mathcal Y_{T^o}/T^o)\xrightarrow{\sim} 
\ker(Sp_{\mathcal E/\tilde f_X}:E^{2p}_{Hdg}(\mathcal E/S)\to\psi_SE^{2p}_{Hdg}(\mathcal Y_{T^o}/T^o)).
\end{eqnarray*}
For $\lambda\in H^{2p}(\mathcal E_{0,\mathbb C,t'}^{an},\mathbb Q)$, $t'\in S(\mathbb C)$, we consider
$\pi_S:\tilde S_{\mathbb C}^{an}\to S_{\mathbb C}^{an}$ the universal covering in $\AnSm(\mathbb C)$ and 
\begin{equation*}
\mathbb V_S(\lambda):=\pi_S(\lambda\times\tilde S^{an}_{\mathbb C})
\subset E^{2p,an}_{DR}(\mathcal E_{0,\mathbb C}/S_{\mathbb C}) 
\end{equation*}
the flat leaf induced by the universal covering (see section 2.1) and
\begin{equation*}
\mathbb V_S^p(\lambda):=\mathbb V_S(\lambda)\cap F^pE^{2p}_{DR}(\mathcal E_{0,\mathbb C}/S_{\mathbb C})
\subset E^{2p}_{DR}(\mathcal E_{0,\mathbb C}/S_{\mathbb C}),
\end{equation*}
which is an algebraic variety by \cite{DCK} (finite over $S$). Note that  
$\mathbb V^p_S(\lambda)\subset HL^{p,2p}(\mathcal E_{0,\mathbb C}/S_{\mathbb C})$
is the union of the irreducible components passing through $\lambda$.
Similarly, for $\eta\in H^{2p}(\mathcal E_{\mathbb C,t'}^{an},\mathbb Q)$, $t'\in S(\mathbb C)$, we consider
$\pi_S:\tilde S_{\mathbb C}^{an}\to S_{\mathbb C}^{an}$ the universal covering in $\AnSm(\mathbb C)$ and 
\begin{equation*}
\mathbb V_S(\eta):=\pi_S(\eta\times\tilde S^{an}_{\mathbb C})
\subset E^{2p,an}_{DR}(\mathcal E_{\mathbb C}/S_{\mathbb C}) 
\end{equation*}
the flat leaf induced by the universal covering (see section 2.1) and
\begin{equation*}
\mathbb V_S^p(\eta):=\mathbb V_S(\eta)\cap F^pE^{2p}_{DR}(\mathcal E_{\mathbb C}/S_{\mathbb C})
\subset E^{2p}_{DR}(\mathcal E_{\mathbb C}/S_{\mathbb C}).
\end{equation*}
We have
\begin{eqnarray*}
H^{2p}E_{Hdg}(\mathcal Y_{T^o}/T^o)=((E^{2p}_{DR}(\mathcal Y_{T^o}/T^o),F),
R^{2p}\tilde f_{X*}\mathbb Q_{\mathcal Y^{an}_{T^o,\mathbb C}},\tilde f_{X*}\alpha(\mathcal Y_{T^o}))
\in MHM_{\mathbb Q,gm}(T^o)
\end{eqnarray*}
By \cite{PS}, the filtered vector bundle
\begin{equation*}
(E^{2p}_{DR}(\mathcal Y_{T^o}/T^o),F)=H^{2p}\tilde f_{X*}E_{zar}(\Omega^{\bullet}_{\mathcal Y_{T^o}/T^o},F_b)
\in\Vect_{\mathcal Dfil}(T^o)
\end{equation*}
extend to a filtered vector bundle
\begin{equation*}
(E^{2p}_{DR,V_S}(\mathcal Y_{T^o}/T^o),F):=
H^{2p}\tilde f_{X*}E_{zar}(\Omega^{\bullet}_{\mathcal Y_{T^o}/T^o}(\log\mathcal E),F_b)
\in\Vect_{\mathcal Dfil}(\bar T^o)
\end{equation*}
such that 
\begin{eqnarray*}
H^{2p}\psi_SE_{Hdg}(\mathcal Y_{T^o}/T^o)=(i_S^{*mod}(E^{2p}_{DR,V_S}(\mathcal Y_{T^o}/T^o),F,W),
(\psi_SR^{2p}\tilde f_{X*}\mathbb Q_{\mathcal Y^{an}_{T^o,\mathbb C}},W),\psi_S\tilde f_{X*}\alpha(\mathcal Y_{T^o}))
\end{eqnarray*}
and
\begin{eqnarray*}
Sp_{\mathcal E/\tilde f_X}:=(Sp_{\mathcal E/\tilde f_X},Sp_{\mathcal E/\tilde f_X}):
((E^{2p}_{DR}(\mathcal E/S),F,W),
(R^{2p}(\tilde f_X\circ i_{\bullet})_*\mathbb Q_{\mathcal E^{an}_{\bullet,\mathbb C}},W),
(\tilde f_X\circ i_{\bullet})_*\alpha(\mathcal E_{\bullet})) \\
\to (i_S^{*mod}(E^{2p}_{DR,V_S}(\mathcal Y_{T^o}/T^o),F,W)=((E^{2p}_{DR,V_S}(\mathcal Y_{T^o}/T^o)\cap p^{-1}(S),F),W), \\
(\psi_SR^{2p}\tilde f_{X*}\mathbb Q_{\mathcal Y^{an}_{T^o,\mathbb C}},W),\psi_S\tilde f_{X*}\alpha(\mathcal Y_{T^o}))
\end{eqnarray*}
where $p:E^{2p}_{DR,V_S}(\mathcal Y_{T^o}/T^o)\to\bar T^o$ is the projection. For 
\begin{equation*}
\alpha\in H^{2p}(\mathcal Y_{t''}^{an},\mathbb Q)
=i_{t'}^*\psi_SR^{2p}\tilde f_{X*}\mathbb Q_{\mathcal Y_{T^o,\mathbb C}^{an}}, \, t'\in S_{\mathbb C}, \, t''\in T^o(\mathbb C), 
\end{equation*}
we consider 
\begin{itemize}
\item $\pi_{T^o}:\tilde T^{o,an}\to T_{\mathbb C}^{o,an}$ the universal covering in $\AnSm(\mathbb C)$, 
the factorization 
\begin{equation*}
\pi_S:\tilde S_{\mathbb C}^{an}\to\tilde{\bar T}_{\mathbb C}^{o,an}\xrightarrow{\pi_{\bar T^o}}\bar T_{\mathbb C}^{o,an}
\end{equation*}
and 
\begin{equation*}
\mathbb V_{T^o}(\alpha):=\pi_{T^o}(\alpha\times\tilde T_{\mathbb C}^{o,an})\subset 
E^{2p,an}_{DR}(\mathcal Y_{T^o,\mathbb C}/T^o_{\mathbb C})
\end{equation*}
the flat leaf (see section 2.1) and
\begin{equation*}
\mathbb V_{T^o}^p(\alpha):=\mathbb V_{T^o}(\alpha)\cap F^pE^{2p}_{DR}(\mathcal Y_{T^o,\mathbb C}/T^o_{\mathbb C})
\subset E^{2p}_{DR}(\mathcal Y_{T^o\mathbb C}/T^o_{\mathbb C}),
\end{equation*}
which is an algebraic subvariety (finite over $T^o$) by \cite{DCK}, 
\item $\pi_S:\tilde S_{\mathbb C}^{an}\to S_{\mathbb C}^{an}$ the universal covering in $\AnSm(\mathbb C)$ and 
\begin{equation*}
\mathbb V_S(\alpha):=\pi_S(\alpha\times\tilde S^{an}_{\mathbb C})\subset 
H^{2p}\psi_SE_{DR}(\mathcal Y_{T^o,\mathbb C}/T^o_{\mathbb C})^{an}=
i_S^{*mod}(E^{2p}_{DR,V_S}(\mathcal Y_{T^o,\mathbb C}/T_{\mathbb C}^o))^{an} 
\end{equation*}
the flat leaf (see section 2.1) and
\begin{equation*}
\mathbb V_S^p(\alpha):=\mathbb V_S(\alpha)\cap F^pH^{2p}\psi_SE_{DR}(\mathcal Y_{T^o,\mathbb C}/T^o_{\mathbb C})^{an}
\subset H^{2p}\psi_SE_{DR}(\mathcal Y_{T^o\mathbb C}^/T^o_{\mathbb C})^{an}.
\end{equation*} 
\end{itemize}
Consider the exact sequence of variation of mixed Hodge structure over $S$
\begin{eqnarray*}
0\to W_{2p-1}E_{Hdg}^{2p}(\mathcal E/S)\to E_{Hdg}^{2p}(\mathcal E/S)
\xrightarrow{(i_i^*)_{0\leq i\leq r}:=q}\oplus_{i=0}^rE_{Hdg}^{2p}(\mathcal E_i/S)\to 0.  
\end{eqnarray*}
It induces the exact sequence of vector bundles over $S$
\begin{eqnarray*}
0\to W_{2p-1}E_{DR}^{2p}(\mathcal E/S)\to E_{DR}^{2p}(\mathcal E/S)
\xrightarrow{(i_i^*)_{0\leq i\leq r}:=q}\oplus_{i=0}^rE_{DR}^{2p}(\mathcal E_i/S)\to 0,  
\end{eqnarray*}
and the exact sequence of presheaves on $S_{\mathbb C}^{an}$,
\begin{eqnarray*}
HL^{2p,p}(\mathcal E_{\mathbb C}/S_{\mathbb C})\xrightarrow{q}
\oplus_{i=0}^rHL^{2p,p}(\mathcal E_{i,\mathbb C}/S_{\mathbb C})\xrightarrow{e} \\ 
J(W_{2p-1}E^{2p}_{Hdg}(\mathcal E_{\mathbb C}/S_{\mathbb C})):=
\Ext^1(\mathbb Q_S^{Hdg},W_{2p-1}E^{2p}_{Hdg}(\mathcal E_{\mathbb C}/S_{\mathbb C})),  
\xrightarrow{\iota}
\Ext^1(\mathbb Q_S^{Hdg},W_{2p-1}E^{2p}_{Hdg}(\mathcal E_{\mathbb C}/S_{\mathbb C}))
\end{eqnarray*}
where 
$HL^{2p,p}(\mathcal E_{\mathbb C}/S_{\mathbb C})\subset E_{DR}^{2p}(\mathcal E_{\mathbb C}/S_{\mathbb C})$, 
$HL^{2p,p}(\mathcal E_{i,\mathbb C}/S_{\mathbb C})\subset E_{DR}^{2p}(\mathcal E_{i,\mathbb C}/S_{\mathbb C})$ 
are the locus of Hodge classes.
Since the monodromy of $R^{2p}\tilde f_{X*}\mathbb Q_{\mathcal Y_{T^o,\mathbb C}}^{an}$ is irreducible
by Picard Lefchetz theory as 
$i_{T^o*}:\pi_1(T_{\mathbb C}^o)\to\pi_1(T_{\mathbb C})\xrightarrow{\sim}\pi_1(\mathbb P(l_X)\backslash\Delta)$ 
is surjective, we have 
\begin{equation}\label{qSp4}
Sp_{\mathcal E/\tilde f_X}(\ker q):=Sp_{\mathcal E/\tilde f_X}(W_{2p-1}E_{DR}^{2p}(\mathcal E/S))=0,
\end{equation}
using the fact that there exists a neighborhood 
$V_{\bar S\cup\bar S'}\subset\bar T_{\mathbb C}$ of $\bar S_{\mathbb C}\cup\bar S'_{\mathbb C}$ in $\bar T_{\mathbb C}$ 
for the usual complex topology such that the inclusion 
$i_{\bar S_{\mathbb C}\cup\bar S'_{\mathbb C}}:
\bar S_{\mathbb C}\cup\bar S'_{\mathbb C}\hookrightarrow V_{\bar S\cup\bar S'}$ admits a
retraction $r:V_{\bar S\cup\bar S'}\to\bar S_{\mathbb C}\cup\bar S'_{\mathbb C}$ which is an homotopy equivalence.
On the other hand, since $\tilde f_X\circ i_0:\mathcal E_0\to S$ is a smooth projective morphism,
$E_{Hdg}^{2p}(\mathcal E_0/S)$ is a variation of pure Hodge structure over $S$ polarized by Poincare duality 
\begin{equation*}
<-,->:=(<-,->,<-,->):
((E_{DR}^{2p}(\mathcal E_0/S),F),R^{2p}\tilde f_{X*}\circ i_{\mathcal E_0*}\mathbb Q_{\mathcal E^{an}_{0\mathbb C}},
\alpha(\mathcal E_0))^{\otimes 2}\to\mathbb Q^{Hdg}_S
\end{equation*}
In particular, by the proof of Deligne semi-simplicity theorem using Schimdt results, we have a splitting of
variation of pure Hodge structure over $S$
\begin{equation}\label{SqSp4}
E_{Hdg}^{2p}(\mathcal E_0/S)=q(\ker Sp_{\mathcal E/\tilde f_X})\oplus q(\ker Sp_{\mathcal E/\tilde f_X})^{\perp,<-,->}, \;
\pi_K:E_{Hdg}^{2p}(\mathcal E_0/S)\to q(\ker Sp_{\mathcal E/\tilde f_X})
\end{equation}
Note that since 
\begin{eqnarray*}
F^pq(\ker Sp_{\mathcal E/\tilde f_X})_{\mathbb C}^{\perp}:=F^p(q(\ker Sp_{\mathcal E/\tilde f_X})_{\mathbb C}^{\perp})=
(F^pq(\ker Sp_{\mathcal E/\tilde f_X})_{\mathbb C})^{\perp,<,>_{|F^pE_{DR}^{2p}(\mathcal E_{0\mathbb C}/S_{\mathbb C})}} \\
\mbox{and} \; F^pE_{DR}^{2p}(\mathcal E_{0\mathbb C}/S_{\mathbb C})=
F^pq(\ker Sp_{\mathcal E/\tilde f_X})_{\mathbb C}\oplus F^pq(\ker Sp_{\mathcal E/\tilde f_X})_{\mathbb C}^{\perp} 
\end{eqnarray*}
by the proof of Deligne semi-simplicity theorem, we have 
\begin{eqnarray*}
F^pq(\ker Sp_{\mathcal E/\tilde f_X})^{\perp}:=F^p(q(\ker Sp_{\mathcal E/\tilde f_X})^{\perp})=
(F^pq(\ker Sp_{\mathcal E/\tilde f_X}))^{\perp,<,>_{|F^pE_{DR}^{2p}(\mathcal E_0/S)}} \\
\mbox{and} \; F^pE_{DR}^{2p}(\mathcal E_0/S)=
F^pq(\ker Sp_{\mathcal E/\tilde f_X})\oplus F^pq(\ker Sp_{\mathcal E/\tilde f_X})^{\perp}. 
\end{eqnarray*}
Let $\lambda\in F^pH^{2p}(E_0^{an},\mathbb Q)$, where we recall $E_0=\mathcal E_{0,\mathbb C,t}$
and $E=\mathcal E_{\mathbb C,t}=\cup_{i=0}^rE_i$. 
Consider then $\tilde\lambda\in H^{2p}(E^{an},\mathbb Q)$, such that $q(\tilde\lambda)=\lambda$, and 
\begin{equation*}
Sp_{\mathcal E/\tilde f_X}(\tilde\lambda)\in Sp_{E/f_X}(H^{2p}(E^{an},\mathbb Q))
\subset i_t^*\psi_SR^{2p}\tilde f_{X*}\mathbb Q_{\mathcal Y_{T^o,\mathbb C}^{an}}. 
\end{equation*}
By (\ref{SqSp}), we have
\begin{equation*}
\lambda=\lambda^K+\lambda^L\in F^pH^{2p}(E_0^{an},\mathbb Q), \;
\lambda^K\in i_t^*F^pq(\ker Sp_{\mathcal E/\tilde f_X})_{\mathbb Q}, \,
\lambda^L\in i_t^*F^pq(\ker Sp_{\mathcal E/\tilde f_X})^{\perp}_{\mathbb Q}.
\end{equation*} 
By (\ref{qSp}), if $\lambda\in i_t^*F^pq(\ker Sp_{\mathcal E/\tilde f_X})^{\perp}_{\mathbb Q}$, 
the locus of Hodge classes passing through $\lambda$ 
\begin{eqnarray*}
\mathbb V_S^p(\lambda):=\mathbb V_S(\lambda)\cap F^pE^{2p}_{DR}(\mathcal E_{0,\mathbb C}/S_{\mathbb C})
\subset E^{2p}_{DR}(\mathcal E_{0,\mathbb C}/S_{\mathbb C}),
\end{eqnarray*}
inside the De Rham vector bundle of $\tilde f_X\circ i_{\mathcal E_0}:\mathcal E_0\to S$ satisfy
\begin{eqnarray}\label{qWSeq4}
\mathbb V_S^p(\lambda)=
q(Sp_{\mathcal E/\tilde f_X}^{-1}(\mathbb V_S^p(Sp_{\mathcal E/\tilde f_X}(\tilde\lambda))))\cap\pi_K^{-1}(0)
\subset E^{2p}_{DR}(\mathcal E_{0\mathbb C}/S_{\mathbb C}),
\end{eqnarray}
where 
\begin{itemize}
\item $\mathbb V_S(\lambda)\subset E^{2p}_{DR}(\mathcal E_{0,\mathbb C}/S_{\mathbb C})$,
$\mathbb V_S(Sp_{\mathcal E/\tilde f_X}(\tilde\lambda))\subset E^{2p}_{DR,V_S}(\mathcal Y_{T^o,\mathbb C}/T_{\mathbb C}^o) $,
are the flat leaves, e.g. $\mathbb V_S(\lambda):=\pi_S(\lambda\times\tilde S^{an}_{\mathbb C})$ where
$\pi_S:H^{2p}(E^{an},\mathbb C)\times\tilde S_{\mathbb C}^{an}\to E^{2p,an}_{DR}(\mathcal E_{\mathbb C}/S_{\mathbb C})$ 
is the morphism induced by the universal covering $\pi_S:\tilde S_{\mathbb C}^{an}\to S_{\mathbb C}^{an}$,
\item $q:=q\otimes\mathbb C:E^{2p}_{DR}(\mathcal E_{\mathbb C}/S_{\mathbb C})\to
E^{2p}_{DR}(\mathcal E_{0,\mathbb C}/S_{\mathbb C})$ is the quotient map.
\item $Sp_{\mathcal E/\tilde f_X}:=Sp_{\mathcal E/\tilde f_X}\otimes\mathbb C:
E^{2p}_{DR}(\mathcal E_{\mathbb C}/S_{\mathbb C})\to i_S^{*mod}E^{2p}_{DR,V_S}(\mathcal Y_{T^o,\mathbb C}/T_{\mathbb C}^o)$.
\end{itemize}
Indeed, by (\ref{qSp4}), we have a factorization 
\begin{eqnarray*} 
Sp_{\mathcal E/\tilde f_X}:E^{2p}_{Hdg}(\mathcal E/S)\xrightarrow{q}E^{2p}_{Hdg}(\mathcal E_0/S) 
\xrightarrow{\Gr_W^{2p}Sp_{\mathcal E/\tilde f_X}}\psi_SE^{2p}_{Hdg}(\mathcal Y_{T^o}/T^o),
\end{eqnarray*}
hence, by (\ref{SqSp4}), for $t'\in S(\mathbb C)$, $\lambda_{t'}\in F^pH^{2p}(\mathcal E_{0\mathbb C,t})$ 
if and only if $Sp_{\mathcal E,\tilde f_X}(\tilde\lambda_{t'})\in F^p\psi_SE_{DR}(\mathcal Y_{T^o}/T^o)_{t'}$.
Now,
\begin{itemize}
\item If $Sp_{\mathcal E/\tilde f_X}(\tilde\lambda)=0$, we have
by lemma \ref{WSp} applied to $\tilde f_X:\mathcal Y_{\bar T^o}\to\bar T^o$, 
$\tilde\lambda=c_{\mathcal E}(\tilde\lambda)$ with 
$\tilde\lambda\in F^pH^{2p-1}_E(\mathcal Y^{an}_{\mathbb C,C},\mathbb Q)$
where $C\subset\bar T_{\mathbb C}^o$ is a smooth transversal slice, hence $\dim(C)=1$ and $\mathcal Y_C$ is smooth, 
such that $C\cap S_{\mathbb C}=\left\{t\right\}$ and 
\begin{equation*}
\mathbb V^p_S(\tilde\lambda)=c_{\mathcal E}(\mathbb V^p_S(\tilde\lambda)), \; 
\mathbb V^p_S(\tilde\lambda)\subset 
E_{DR,\mathcal E_{\mathbb C}}^{2p-1}(\mathcal Y_{\bar T^o,\mathbb C}/\bar T_{\mathbb C}^o)
\end{equation*}
with 
\begin{eqnarray*}
c_{\mathcal E}=(c_{\mathcal E},c_{\mathcal E}):
((E_{DR,\mathcal E}^{2p-1}(\mathcal Y_{\bar T^o}/\bar T^o),F,W),
(R^{2p-1}\tilde f_{X*}\circ\Gamma_{\mathcal E}\mathbb Q_{\mathcal Y_{\bar T^o,\mathbb C}^{an}},W),
\alpha(\mathcal Y_{\bar T^o})) \\
\to ((E_{DR}^{2p}(\mathcal E/S),F,W),
(R^{2p}(\tilde f_X\circ i_{\mathcal E})_*\mathbb Q_{\mathcal E_{\mathbb C}^{an}},W),
\alpha(\mathcal E_{\bullet}))
\end{eqnarray*}
We have for each $t'\in S(\mathbb C)$, the isomorphism
\begin{eqnarray*}
\oplus_{card I=2}F^pH^{2p}(E^{an}_{I\mathbb C,t'},\mathbb Q)=
F^p\Gr^W_{2p}H^{2p-2}((\bar{\mathcal Y_{\mathbb C,C'}}\backslash E)^{an},\mathbb Q)
\xrightarrow{c(\mathbb Z(\mathcal Y_C\backslash\mathcal E_{\mathbb C,t'}))} \\
F^p\Gr^W_{2p}H^{2p-1}_E(\bar{\mathcal Y^{an}_{\mathbb C,C'}},\mathbb Q)=
F^p\Gr^W_{2p}H^{2p-1}_E(\mathcal Y^{an}_{\mathbb C,C'},\mathbb Q)
\end{eqnarray*}
where $C\subset\bar T_{\mathbb C}^o$ is a smooth transversal slice, hence $\dim(C')=1$ and $\mathcal Y_{C'}$ is smooth, 
such that $C'\cap S_{\mathbb C}=\left\{t'\right\}$ 
and $\bar{\mathcal Y_{C'}}\in\PSmVar(\mathbb C)$ is a smooth compactification of $\mathcal Y_{C'}$,
the last equality follows from excision (note that $E$ is proper).
Hence, assuming the Hodge conjecture for smooth projective varieties of dimension less or equal to $d-1$,
\begin{equation*} 
\mathbb V^p_S(\lambda)\subset 
\Gr^W_{2p}E_{DR,\mathcal E_{\mathbb C}}^{2p-1}(\mathcal Y_{\bar T^o,\mathbb C}/\bar T_{\mathbb C}^o), \;
\mathbb V^p_S(\lambda)=c_{\mathcal E}(\mathbb V^p_S(\lambda))
\subset E^{2p}_{DR}(\mathcal E_{0\mathbb C}/S_{\mathbb C})
\subset\Gr^W_{2p}E^{2p}_{DR}(\mathcal E_{\mathbb C}/S_{\mathbb C})
\end{equation*}
are defined over $\bar{\mathbb Q}$ and its Galois conjugates are also components of the locus of Hodge classes,
since $\dim(E_I)=d-1$ for $I\subset[0,\cdots,r]$ such that $card I=2$. 
\item Consider now the case where  $\lambda\in i_t^*F^pq(\ker Sp_{\mathcal E/\tilde f_X})^{\perp}_{\mathbb Q}$.
We have the key equality
\begin{equation*}
\mathbb V_S^p(Sp_{\mathcal E/\tilde f_X}(\tilde\lambda))=
\overline{\mathbb V_{T^o}^p(Sp_{\mathcal E/\tilde f_X}(\tilde\lambda))}\cap p^{-1}(S_{\mathbb C})
\subset E^{2p}_{DR,V_S}(\mathcal Y_{T^o,\mathbb C}/T^o_{\mathbb C})
\end{equation*}
where $p:=p\otimes\mathbb C:E^{2p}_{DR,V_S}(\mathcal Y_{T^o,\mathbb C}/T^o_{\mathbb C})\to\bar T^o_{\mathbb C}$
is the projection, the inclusion $\subset$ is obvious whereas the inclusion $\supset$ follows from \cite{DCK} lemma 2.11
which state the invariance of $\tilde\lambda$ under the monodromy 
\begin{equation*}
\Im(\pi_1(p(\mathbb V_{T^o}^p(\tilde\lambda))\cap D^*)\to\pi_1(D^*)), \;
(D,D^*)\to(\bar T^{o,an}_{\mathbb C},\bar T^{o,an}_{\mathbb C}\backslash S_{\mathbb C}^{an}).
\end{equation*}
But by \cite{B8} theorem 1, the Hodge conjecture holds for projective hypersurfaces hence 
\begin{equation*}
\mathbb V_{T^o}^p(Sp_{\mathcal E/\tilde f_X}(\tilde\lambda))\subset 
E^{2p}_{DR}(\mathcal Y_{T^o,\mathbb C}/T^o_{\mathbb C})
\end{equation*}
is an algebraic subvariety defined over $\bar{\mathbb Q}$ 
and its Galois conjugates are also components of the locus of Hodge classes. Hence, using (\ref{qWSeq4}) 
\begin{equation*}
\mathbb V_S^p(\lambda)=
q(Sp_{\mathcal E/\tilde f_X}^{-1}(\mathbb V_S^p(Sp_{\mathcal E/\tilde f_X}(\tilde\lambda))))\cap\pi_K^{-1}(0)
\subset E^{2p}_{DR}(\mathcal E_{0\mathbb C}/S_{\mathbb C}) 
\end{equation*}
is an algebraic subvariety defined over $\bar{\mathbb Q}$ 
and its Galois conjugates are also components of the locus of Hodge classes.
\end{itemize}
Hence, 
\begin{itemize}
\item $\mathbb V_S^p(\lambda_K)\subset E^{2p}_{DR}(\mathcal E_{0\mathbb C}/S_{\mathbb C})$ 
is an algebraic subvariety defined over $\bar{\mathbb Q}$ 
and its Galois conjugates are also components of the locus of Hodge classes gives, thus by \cite{B7} theorem 4, 
\begin{equation*}
\lambda^K=[Z^K]\in H^{2p}(E^{an}_0,\mathbb Q), \; Z^K\in\mathcal Z^p(E_0),
\end{equation*}
\item $\mathbb V_S^p(\lambda_S)\subset E^{2p}_{DR}(\mathcal E_{0\mathbb C}/S_{\mathbb C})$ 
is an algebraic subvariety defined over $\bar{\mathbb Q}$ 
and its Galois conjugates are also components of the locus of Hodge classes, thus by \cite{B7} theorem 4, 
\begin{equation*}
\lambda^L=[Z^L]\in H^{2p}(E^{an}_0,\mathbb Q), \; Z^L\in\mathcal Z^p(E_0).
\end{equation*}
\end{itemize}
Hence, 
\begin{equation*}
\lambda=\lambda^K+\lambda^L=[Z]\in H^{2p}(E^{an}_0,\mathbb Q), \; Z:=Z^K+Z^L\in\mathcal Z^p(E_0). 
\end{equation*}
Since $p\in\mathbb Z$ and $\lambda\in F^pH^{2p}(E_0^{an},\mathbb Q)$ are arbitrary, the Hodge conjecture holds for $E_0$. 
Proposition \ref{birHdg}(ii) and the fact that
the Hodge conjecture is true for smooth complex projective varieties of dimension less or equal to $d-1$
by induction hypothesis then implies that Hodge conjecture holds for $X$,
since $X$ is birational to $E_0$ and $X,E_0\in\PSmVar(\mathbb C)$ are connected of dimension $d$.
\end{proof}


LAGA UMR CNRS 7539 \\
Universit\'e Paris 13, Sorbonne Paris Cit\'e, 99 av Jean-Baptiste Clement, \\
93430 Villetaneuse, France, \\
bouali@math.univ-paris13.fr

\end{document}